\documentclass[12pt]{amsart}
 \usepackage{amsmath,amsthm,amsfonts,amssymb}
\usepackage{verbatim}
\usepackage{latexsym} 

\usepackage{etoolbox}

\newtheorem{theorem}{Theorem}[section]
\newtheorem{lemma}[theorem]{Lemma}
\newtheorem{proposition}[theorem]{Proposition}
\newtheorem{definition}[theorem]{Definition}
\newtheorem{corollary}[theorem]{Corollary}

\newtheorem{remark}[theorem]{Remark}

\newcommand\RR{{\mathbb{R}}}
\newcommand\CC{{\mathbb{C}}}

\newcommand\NN{\mathbb{ N}}
\newcommand\ZZ{\mathbb{ Z}}
\newcommand{\TT}{\mathbb{ T}}

\renewcommand{\leq}{\leqslant}
\def\({\left(}
\def\){\right)}
\def\[{\left[}
\def\]{\right]}
\def\={\quad = \quad}
\def\+{\quad + \quad}

\def\R{{\mathbb{R}}}
\def\Z{{\mathbb{Z}}}

\def\NN{{\mathbb{N}}}
\def\C{{\mathbb{C}}}

\def\T{{\mathbb{T}}}

\def\em{\endmatrix\right]}

\def\N{\mathbb N}
\def\n{\mathfrak{n}}
\def\ff{if and only if }

\def \n {\mathfrak n}

\def \eO {\eusm O}
\def \a {\alpha}

\def \l {\lambda}

\def \z {\mathfrak z} 

\def \p {\mathfrak p}
\def \eusm {\mathcal}

\def \G {\Gamma}
\def \H {\mathcal H}

\def \S {\mathcal S}

\newcommand{\note}[2][\null]{%
  \marginpar{\renewcommand{\baselinestretch}{1}\vspace{-1em}\hrule\vspace{3pt}%
  \footnotesize\raggedright\textsf{#2\ifx#1\null\else\\\hfill--- 
  {\em #1}\fi}\vspace{1.5em}}%
}

\begin{document}
\title{ Shift-invariant  spaces on $SI/Z$   Lie groups}
\author{ by 
Bradley Currey, Azita Mayeli, Vignon Oussa }

\date{\today}

\maketitle 
\begin{abstract}
Given a simply connected nilpotent Lie group having unitary irreducible representations that are square-integrable modulo the center (SI/Z), we develop a notion of periodization on the group Fourier transform side, and use this notion to give a characterization of shift-invariant spaces in $L^2(N)$ in terms of range functions. We apply these results to study the structure of frame and Reisz families for shift-invariant spaces.  We illustrate these results for the Heisenberg group as well as for other groups with SI/Z representations. 
\end{abstract}

keywords:{ \it Shift-invariant spaces, translation invariant spaces, 
nilpotent Lie groups;  range function; periodization operator; fibers; frame and Reisz bases; }

\section{introduction}
 
Let $G$ be a locally compact abelian (LCA) group, and let $H$ be a discrete subgroup of $G$ for which $G/H$ is compact. For a function $\phi : G \rightarrow \C$ write $L_g\phi = \phi(g^{-1} \cdot), g \in G$. A shift invariant space (SIS) is a closed subspace $\H$ of $L^2(G)$ that is invariant under the action of $H$ by the (unitary) operators $L_h$. The case where $G = \R^d$ has been studied extensively in the literature and plays a central role in the development of approximation theory and nonorthogonal expansions; see for example \cite{BVR1, BVR2, B}. The main idea is due to Helson \cite{H}, whereby, via periodization, an SIS  $\H$ corresponds to a measurable {\it range function}. For general LCA groups, this idea is extended very successfully in \cite{CP} to give a characterization of SIS spaces exactly as in the Euclidean case; see also \cite{KR}. 

In this article we generalize the study of SIS in a different direction, to a natural class of non-abelian Lie groups called SI/Z groups. Let $N$ be a connected, simply connected nilpotent Lie group. Following \cite{CORW}, we say that $N$ is an SI/Z group if almost all of its irreducible representations are square-integrable modulo the center of the group. The effect of the SI/Z condition condition is that for this class of groups, the operator-valued Plancherel transform retains certain key features of the Euclidean case, and in particular, makes it possible to define a useful notion of periodization. Thus, techniques of representation theory, operator theory, spectral methods, Fourier analysis, and approximation theory, are related in a natural setting. We remark that the class of SI/Z groups is broad and in particular contains groups of an arbitrarily high degree of non-commutativity.

 
 The outline of the paper is as follows. After the introduction, we introduce the class of SI/Z groups and some facts about these groups in Section \ref{Notations and Preliminaries}. In Section \ref{shift-invariant spaces} we study shift-invariant subspaces for our setting and their characterization in terms of range functions using group Fourier transform techniques. We apply these results to characterize bases for shift-invariant subspaces in Section \ref{Frames and Bases}. We illustrate our results on some examples of such groups, including the Heisenberg group, in Section \ref{Examples and Applications}.

  \section{Notations and Preliminaries}\label{Notations and Preliminaries}

  Let $N$ be a connected, simply connected nilpotent Lie group with Lie algebra $\n$. Recall that $N$ acts naturally on $\n$ by the adjoint representation and on the linear dual  $\n^*$ of $\n$ by the coadoint representation. We denote these actions multiplicatively, and for $l \in \n^*$, denote by $N(l)$ the stabilizer of $l$ and $\n(l)$ its Lie algebra. The exponential mapping $\exp : \n \rightarrow N$ is a bijection, the center $\z$ of the Lie algebra of $\n$ is non-trivial, and $Z = \exp \z$ is the center of $N$. We now recall certain classical results concerning the unitary representations of such groups originally due to J. Dixmier, A. Kirillov, and L. Pukanszky; citations for their work can be found in standard reference \cite{CORW}.
  
  Let $\pi$ be an irreducible unitary representation of $N$. Then there is an analytic subgroup $P$ of $N$, and a unitary character $\chi$ of $P$, such that $\pi$ is unitarily isomorphic with the representation induced by $\chi$. Writing $P = \exp \p$ and $\chi(\exp Y) = e^{2\pi i l(Y)}$ where $l \in \n^*$, we have $[\p,\p] \subseteq \ker l$. On the other hand, given any $l\in \n^*$, a subalgebra $\p$  of $\n$ is said to be subordinate to $l$ if $[\p,\p] \subseteq \ker l$, and in this case $l$ defines a character $\chi_l$ of $P = \exp \p$ as above.  The representation induced from $\chi_l$ is irreducible if and only if $\p$ is has maximal dimension among subordinate subalgebras, and if $\p$ and $\p'$ are both maximal subordinate subalgebras for $l$, then the induced representations are isomorphic. It follows that there is a canonical bijection between the space $\hat N$ of equivalence classes of irreducible representations of $N$, and the quotient space $\n^* / N$ of coadjoint orbits in $\n^*$. This bijection is a Borel isomorphism, and the Plancherel measure class for $N$ is supported on the (dense, open) subset $\hat N_{\text{max}}$ of $\hat N$ corresponding to the coadjoint orbits of maximal dimension. 
 For an irreducible representation $\pi$ of $N$, we denote the coadjoint orbit corresponding to the equivalence class of $\pi$ by $\eO_\pi$.

  An  unitary irreducible representation of a unimodular Lie group $N$ is said to be {\it square integrable}, or a {\it discrete series representation},  if $s \mapsto \langle \pi(s)u,v\rangle$ belongs to $L^2(N)$ for any vectors $u$ and $v$. If  $N$ is simply connected nilpotent, then the fact that the center is non-trivial implies that there are no discrete series representations. An irreducible representation $\pi$  is {\it square-integrable modulo the center} if 
   $$
\int_{N/Z}\left\vert \left\langle \pi\left(  n\right)  u,v\right\rangle
\right\vert ^{2}d\dot n  <\infty.
$$
We denote by $SI/Z$ the subset of $\hat N$ consisting of those (equivalence classes of) irreducible representations that are square-integrable modulo the center. For many unimodular groups the set $SI/Z$ is empty. In the case where $N$ is nilpotent simply connected, there is an orbital characterization of $SI/Z$-representations, first proved in \cite{MW} and also found in \cite{CORW}. 

\begin{theorem} Let $\pi$ be an irreducible representation of a connected, simply connected nilpotent Lie group $N$. Then the following are equivalent. 

\begin{itemize}

\item[(i)] $\pi$ is square-integrable modulo the center. 

\item[(ii)] For any $l \in \eO_\pi$, $ \n(l) = \z$ and $\eO_\pi = l + \z^\perp$.
\end{itemize}
Moreover, if $SI/Z \ne \emptyset$, then 
$SI/Z = \hat N_{\text{max}}$.  

\end{theorem}


We say that a connected simply connected nilpotent Lie group $N$ is an $SI/Z$ group if  $SI/Z = \hat N_{\text{max}}$. We remark that for each positive integer $s$, there is a  $SI/Z$ group $N$ of step $s$ (see  \cite{CORW}). The abelian case is just $\R^n$, and the simplest two-step example is the Heisenberg group, where $\hat N_{\text{max}}$ consists of the Schr\"odinger representations. From now on, we will assume that $N$ is a $SI/Z$ group.


The preceding shows that $\hat N_{\text{max}}$ is parametrized by a subset of $\z^*$. Indeed, let $\pi \in \hat N_{\text{max}}$. Then $\dim \eO_\pi  = n - \dim \z$, and since $\eO_\pi$ is naturally a symplectic manifold, its dimension is even and we write $\dim \eO_\pi = 2d$. Then Schur's Lemma says that the restriction of $\pi$ to $Z$ is a character of $Z$, and hence $\pi$ determines a unique element $\l = \l_\pi \in \z^*$ so that 
$$
\pi(z) = e^{2\pi i \langle \l,\log z\rangle} I,
$$
where $I$ is the identity operator for the Hilbert space of a realization of $\pi$.  It follows that $\eO_\pi = \{ l \in \n^* : l |_\z = \l\}$. The description of orbits for $\hat N_{\text{max}}$ shows that $\pi \mapsto \l_\pi$ is injective.

An explicit Plancherel transform is obtained by describing the set $\Sigma = \{ \l_\pi : \pi \in \hat N_{\text{max}} \}$, and explicit maximal subordinate subalgebras $\p(\l), \l \in \Sigma$, that vary smoothly with $\l$.
Fix a basis $\{X_1, X_2, \dots , X_n\}$ for $\n$ for which  $\n_j := \text{span}\{X_1, \dots , X_j\}$ is an ideal in $\n$, $1 \le j \le n$, and for which $\n_r = \z$ for some $r$. Having fixed such a basis, we identify $\z$ with $\R^r$ and  $\z^*$ with the subspace $\{ l\in \n^* : l(X_j) = 0, r < j \le n\}$. 
Define $\mathbf{Pf}(l), l \in \n^*$ so that 
$|Pf(l)|^2 = \text{det} M(l)$
where $M(l)$ is the skew-symmetric matrix
$$
M(l) = \bigl[ l [X_i,X_j]_{r \le i,j \le n}\bigr].
$$
The following can be gleaned from \cite{Puk} and \cite{MW}; again a useful reference is \cite{CORW}. 

\begin{proposition} For an $SI/Z$ group, we have the following. 

\begin{itemize}

\item[(i)] $\mathbf{Pf} $ is constant on each coadjoint orbit and $[\pi ] \in SI/Z$ \ff $\mathbf{Pf}(l) \ne 0$ on $\eO_\pi$. 

\item[(ii)] $\Sigma := \{ \l \in \z^* : \mathbf{Pf}(\l) \ne 0\}$ is a cross-section for coadjoint orbits of maximal dimension.

\item[(iii)] Fix $\l \in \Sigma$. Then
$$
\p(\l) := \sum_{j=1}^n \n_j(\l |_{\n_j})
$$
is a maximal subordinate subalgebra for $\l$ and the corresponding induced representation $\pi_\l$ is realized naturally in $L^2(\R^d)$. 

\item[(iv)] For $\phi \in L^1(N) \cap L^2(N)$, the Fourier transform
$$
\hat\phi(\l) := \int_N \phi(x) \pi_\l(x) dx, \ \l \in \Sigma,
$$
implements an isometric isomorphism -- the Plancherel transform --
$$
\mathcal F : L^2(N) \rightarrow L^2\bigl(\z^*, \mathcal {HS}(L^2(\R^d)), |\mathbf{Pf}(\l)| d\l\bigr)
$$
where $ \mathcal {HS}(L^2(\R^d))$ is the Hilbert space of Hilbert-Schmidt operators on $L^2(\R^d)$, and $d\l$ is a suitably normalized Lebesgue measure on $\z^*$ . 

\end{itemize}

\end{proposition}

    \section{Shift invariant spaces}\label{shift-invariant spaces}

 Let $N$ be an $SI/Z$ group.   
   We retain the notations of the preceding section, and recall the basis $\{X_1, X_2, \dots , X_n\}$ chosen above. Identify the center $\z$ of $\n$ with $\R^r$ via the ordered basis $\{X_1, X_2, \dots , X_r\}$ ($r = n - 2d$), and identify the center $Z$ of $N$ with $\R^r$ by 
   $$
z = \exp z_1 X_1 \exp z_2 X_2 \cdots \exp z_r X_r, \ (z_1,z_2, \dots , z_r) \in \R^r.
$$
    Write 
$$
x = \exp x_1 X_{r+1} \exp x_2 X_{r+2} \cdots \exp x_{2d} X_n, \  (x_1,x_2, \dots , x_r) \in \R^{2d}
$$
so that the subset $\mathcal X =\exp \R X_{r+1} \exp \R X_{r+2} \cdots \exp\R X_n$ of $N$ is identified with $\R^{2d}$. For each $n\in N$, we have unique $x \in \mathcal X$ and $z \in Z$  such that $n = xz$ and $N$ is thus identified with $ \R^{2d} \times \R^r$. The dual $\z^*$ is identified with $\R^r $ via the dual basis $\{X_1^*, \dots , X_r^*\}$.  

Let $\phi \in L^1(N)\cap L^2(N)$. For $x \in \mathcal X$ and $z \in Z$, the Plancherel transform of the left translate $L_{xz}\phi$ is the operator-valued function on $\Sigma$ defined at each $\l \in \Sigma$ by 
$$
\mathcal F{\bigl(L(xz) \phi\bigr)}(\l) = e^{2\pi i \langle \l,  z\rangle} \pi_\l(x) \hat\phi(\l).
$$
We denote the $r$ dimensional torus by $\mathbb{T}^r$ and identify it with $[0, 1)^r$ in $\R^r$ as usual. 
For simplicity of notation we write $\mathcal H = \mathcal {HS}(L^2(\R^d))= L^2(\RR^d)\otimes L^2(\RR^d)^*$, and let $\mathcal{L}$ denote  the Hilbert space
$l^{2}\left(  \mathbb{Z}^{r},\mathcal{H}  \right)$ with the norm  
$$
\left\Vert h\right\Vert _{\mathcal{L}}^{2}=\sum_{j\in\mathbb{Z}^{r}%
}\| h_j\|^2_{\mathcal{H} }.
$$

Define the periodized Plancherel transform map $T:L^{2}\left(  N\right)  \rightarrow L^{2}\left(  \mathbb{T}^r, \mathcal{L}%
\right)  $ as follows. Denote by $\mathcal C$ the set of all $\mathcal H$-valued $L^2$-functions on $\R^r$ that are continuous on each set $\T^r + j, j \in \Z^r$. Note that $\mathcal C$ determines a subspace of $L^2(\R^r, \mathcal H)$, and we denote this subspace also by $\mathcal C$. Given $f \in \mathcal C$, 
 define the sequence-valued function $a_f$ on $\T^r$ by 
$$
a_f(\sigma) = \bigl( f(\sigma + j)\bigr)_{j \in \Z^r} ,\quad \sigma\in \TT^r. 
$$
Then it is easy to check that $a_f \in C(\T^r, \mathcal L)$ and we have
\begin{equation}\label{per}
\int_{\R^r} \| f(\l)\|^2 d\l = \sum_{j \in \Z^r} \int_{\T^r} \|f(\sigma+ j)\|_{\mathcal H}^2 d\sigma = \int_{\T^r} \sum_{j \in \Z^r}\|a_f(\sigma)_j\|_{\mathcal H}^2 d\sigma . 
\end{equation}
 
Hence for a.e. $\sigma \in \T^r$, $a_f(\sigma) \in \mathcal L$, the function $a_f$ is an $L^2$, $\mathcal L$-valued function on $\T^r$, and $\| f\| = \| a_f\|$.  $\mathcal C$ is dense in $L^2(\R^r, \mathcal H)$, and so we   extend this mapping to an isometry $A : L^{2}( \R^r, \mathcal{H} )\rightarrow L^2(\T^r, \mathcal L)$ given by $f\mapsto Af(\sigma)=a_f(\sigma)$. 
 Recalling the Pfaffian $\mathbf{Pf}\left(  \lambda\right)$ and identifying $\z^*$ with $\RR^{r}$, let $M : L^{2}( \RR^{r}, \mathcal{H}, |\mathbf{Pf}(\l)| d\l ) \rightarrow  L^{2}( \R^r, \mathcal{H} )$ be the natural isometric isomorphism given by
 $$
M f(\l) = f(\l) | \mathbf{Pf}(\l) | ^{1/2}
$$
and put $T = A \circ M \circ \mathcal F$. The diagram below displays the composition of isometries as described.
$$
L^{2}\left( N\right) \overset{\mathcal{F}}{\rightarrow }L^{2}\left( 
 \RR^{r},\mathcal{H},\mathbf{Pf}\left( \lambda \right) d\lambda
\right) \overset{M}{\rightarrow }L^{2}\left( \RR^{r},\mathcal{H}%
,d\lambda \right) \overset{A}{\rightarrow }L^{2}\left( \mathbb{T}^{r},%
\mathcal{L}\right).
$$
 Finally, for $n \in N$, we define the unitary operator  $\tilde\pi_\sigma(n) :  \mathcal L \rightarrow  \mathcal L$ by 
$$
\bigl(\tilde\pi_\sigma(n)h\bigr)_j = \pi_{\sigma+j}(n) \circ h_j, \ \ h \in \mathcal L
$$
and define $\tilde\pi(n) : L^2(\T^r, \mathcal L) \rightarrow L^2(\T^r, \mathcal L)$ by 
$$
\bigl(\tilde\pi(n)a\bigr)(\sigma) =\tilde\pi_{\sigma}(n)a(\sigma), \ \ a \in L^2(\T^r, \mathcal L).
$$
Note that if $z\in Z$, then $\bigr(\tilde\pi(z)a\bigl)(\sigma) = e^{2\pi i \langle\sigma, z\rangle}  a(\sigma)$  for all  $ a\in L^2(\T^r, \mathcal L)$. 
We have almost already proved the following. 
\begin{lemma}
The mapping $T$ is an isometric
isomorphism, and for each $n\in N$, 

\begin{equation} \label{intertwining-translation-1}
T \bigl(L_{n}\phi\bigr)(\sigma) = \bigl(\tilde\pi(n) T\phi\bigr)(\sigma).
\end{equation}
\end{lemma}

\begin{proof} 
We show that $A$ is surjective. Take $a \in C(\T^r, \mathcal L)$, and define 
 $f : \R^r \rightarrow \mathcal H$ by $f(\l) =a(\sigma + j) =  a(\sigma)_j$ where $\sigma:=\l-j \in \TT^r$ and $j \in \ZZ^r$. Then $f \in \mathcal C$ and the identity (\ref{per}) shows that $f $ is an $L^2$-function on $\R^r$ with $A f = a$. By density of $\mathcal C$ and $C(\T^r,\H)$, we have that $A$ is surjective. It follows that $T$ is unitary.
 
 \end{proof}

Next we turn to the definition of shift-invariant spaces and range functions. Recall that both $\z$ and $\z^*$ are identified with $\R^r$ via the chosen basis $X_1, \dots , X_r$ for $\z$.  Denote by $\G_0$ 
the lattice of integral points in $Z$. Then $Z_{\Z^r}$ is identified by $\ZZ^r$. 

 For any discrete subset $\Gamma_1$ of $\mathcal X$,  put 
$$
\Gamma = \{ xz \in N : x \in \Gamma_1, z \in \G_0\}. 
$$\\
 {\it Remark:} The choice of basis for $\z$ is completely arbitrary, and given any lattice $\G_0$ in $Z$, there is a basis for $\z$ for which $\G_0$ is the lattice of integral points.\\

 
 \noindent
We say a  closed subspace $\S$ of $L^{2}\left(  N\right)  $ is $\Gamma$\textit{-shift
invariant} if for any $\gamma\in \Gamma$ and $\phi\in \S$ 
$$
L_\gamma\phi\in \S 
$$
 \begin{definition}\label{rangef}
A {\it range function} is  a mapping from $\T^r$ into the set of closed subspaces of the Hilbert space  ${\mathcal L}$.
 For any $\sigma\in \T^r$, we call $J(\sigma)$ {the fiber space} associated to $\sigma$. 
 We say a range function is measurable if for any $a\in L^2(\T^r,\mathcal L)$, the mapping $\sigma \mapsto \langle P_\sigma u, v\rangle, ~ \forall u, v\in \mathcal L$  is measurable, where $P_\sigma$ is the orthogonal projector of $\mathcal L$ onto $J(\sigma)$. 
 \end{definition}
 
 For a measurable range function $J$, the condition ``$a(\sigma) \in J(\sigma)$ for a.e. $\sigma$''  determines a subspace $\mathcal M_J$  of $L^2(\T^r, \mathcal L)$:
  $$
 \mathcal M_J = \{ a \in L^2(\T^r, \mathcal L) : a(\sigma) \in J(\sigma) \text{ for a.e. } \sigma\}.
 $$
It is easily seen that $\mathcal M_J$ is closed in $L^2(\T^r, \mathcal L)$. 
 Observe that for range functions $J, J'$, $\mathcal M_J = \mathcal M_{J'}$ if and only if $J(\sigma) = J'(\sigma)$ for a.e. $\sigma$. In general we identify two range functions that are equal a.e.

\begin{lemma}\label{from integral to closed subspace}
Let $J$ be a range function and put $\S=T^{-1}(\mathcal M_J)$. If 
$$
\tilde\pi_\sigma(\Gamma_1)\bigl(J(\sigma)\bigr) \subseteq J(\sigma)
$$
holds for a.e. $\sigma \in \T^r$, then $\S$ is $\Gamma$-shift invariant.
 
\end{lemma}

\begin{proof}  For each $\phi\in \S$ and $k\in \Gamma_1$, we apply  (\ref{intertwining-translation-1})  to see that for a.e. $\sigma \in \T^r$, $k \in \Gamma_1, m \in \G_0$,
$$
T\bigl( L_{km} \phi\bigr)(\sigma) = e^{2\pi i\langle \sigma, m\rangle} \tilde\pi_\sigma (k)\bigl(T\phi(\sigma)\bigr)
$$
and so by definition of $\mathcal M_J$, $T\bigl( L(\gamma) \phi\bigr) \in \mathcal M_J$ and $L_\gamma\phi \in \S$ for all $\gamma \in \Gamma$.

  \end{proof}

 The characterization of $\Gamma$-shift-invariant subspaces is given in 
  
 \begin{theorem}\label{subspace-to-integral} 
 
  Let $\S\subseteq L^2(N)$ be a closed subspace. Then the following are equivalent. 
  \begin{itemize} 
  \item[(i)] $\S$ is  $\Gamma$-shift invariant.
  \item[(ii)] There is a unique range function $J$ up to equivalency  such that $J(\sigma) $ is $\tilde\pi_\sigma(\Gamma_1)$-invariant for a.e. $ \sigma$, and $T(\S) = \mathcal M_J$.
 
  \end{itemize}
 \end{theorem}
 \begin{proof} 
  To show  (i) $ \implies $ (ii), we apply \cite[ Theorem 8]{H}. Suppose that $\S$ is shift invariant and let $\phi \in \S$, $a = T\phi$. For each $m \in \G_0$ and $\sigma \in \Sigma$, the relation (\ref{intertwining-translation-1}) shows that 
  $$
  e^{2\pi i m\cdot \sigma} a(\sigma) = T\bigl( L_m\phi\bigr)(\sigma)
  $$
  so that $T(\S)$ is doubly invariant in the sense of \cite{H}. By \cite[Theorem 8]{H},  there exists a range function $J$ such that $T(\S) = \mathcal M_J$. It remains to show that $J(\sigma)$ is $\tilde\pi_\sigma(\Gamma_1)$-invariant for a.e. $\sigma$. 
  
  Choose an orthonormal basis  $\{E^{(n)}\}_{n \in \N}$ for $\mathcal L$. For each $n \in \N$ and $p \in \Z$ put  $G_{p,n}(\sigma)= e^{2\pi i\langle\sigma ,p\rangle}E^{(n)}$. Then $\{ G_{p,n} : n \in \N, p \in \Z\}$ is an orthonormal basis for  $L^2(\TT^r, \mathcal L)$.  Let $P:=P_{\mathcal T}$ denote the orthogonal projector of $L^2(\TT^r, \mathcal L)$ onto $T(\S)$, and for each $n,p$, choose a function $F_{n,p} $ that belongs to the equivalence class of $P(G_{n,p})$. (From the proof of \cite[Theorem 8]{H} we see that $J(\sigma):=\overline{\text{span}}\{F_{n,p}(\sigma):~ n\in\N, p \in \Z\}$ for all $\sigma$.) 
 Now by shift-invariance of $\S$, definition of $\mathcal M_J = T(\S)$ and  (\ref{intertwining-translation-1}), for each $n\in \N, p\in \Z$, and $k \in \gamma'$, we have a conull subset $E_{n,p, k}$ of $\T^r$ such that for all $\sigma \in E_{n,p, k}$, the sequence
  $$
  \tilde\pi_\sigma(k) F_{n,p}(\sigma) = TL_k T^{-1}F_{n,p}(\sigma)
  $$
belongs to $J(\sigma)$. Put $E = \cap\{E_{n,p,k} : n\in \N, p\in \Z, k\in \Gamma'\}$. Then $E$ is conull, and for $\sigma \in E$, we have $ \tilde\pi_\sigma(k) F_{n,p}(\sigma)$ belongs to $J(\sigma)$ for all $n, p, $ and $k$. Since $J(\sigma)$ is spanned by $F_{n,p}(\sigma)$,then $J(\sigma)$ is $\tilde\pi_\sigma(\Gamma')$-invariant.  The proof of $(ii)\implies (i)$ is obtained by  Lemma  \ref{from integral to closed subspace}.

    \end{proof}

 \section{Frames and Bases}\label{Frames and Bases}   
 
 \subsection{Frames}
 Let $X =\{ \eta_\a\}$ be a countable family of vectors in a Hilbert space $\H$. Recall that $X$ is a Bessel family if there is a positive constant $B $ such that
 $$
 \sum_{\a} |\langle h, \eta_\a\rangle |^2 \le B\ \| h\|^2
 $$
 holds for all $h \in $ span$X$. If in addition there is $0 < A \le B<\infty$ so that 
 $$
 A\ \|h\|^2 \le  \sum_{\a }|\langle h, \eta_\a\rangle |^2 \le B\ \| h\|^2
 $$
 holds for all $h \in $ span$X$, then we say that $X$ is a frame (for its span). Finally, $X$ is  a Riesz family with  positive finite constants $A$ and $B$ if 
 $$
 A \sum_{\a} |a_\a|^2 \le \left\|\sum_{\a}a_\a \eta_\a \right\|^2 \le B \sum_{\a} |a_\a|^2
 $$
 holds for all finitely supported indexed sets $(a_\a)_{\a}$ of complex numbers. If a Riesz family is complete in $\H$ we say that it is a Riesz basis. If   $A = B = 1$ for a Riesz basis $X$, then $X$ is an orthogonal family and $\|\eta_\a\|=1$ for all $\a$. 
 
 Fix a discrete subset $\Gamma $ of $N$ of the form $\Gamma_1 \G_0$. 
  Let $\mathcal A\subset L^2(N)$ be a countable set.  Define
  $$
  E(\mathcal A)=\{L_\gamma \phi: ~ \gamma\in \Gamma, \phi\in \mathcal A\}
  $$ 
  and put
$\S =\overline{\text{span}}\ E(\mathcal A)$. Let  $J$ be the range function  associated to $\S$. With these definitions we have

\begin{theorem}\label{characterization of frames} The system 
$E(\mathcal A)$ is a frame with constants $0<A\leq B<\infty$ (or a Bessel family with constant $B$) if and only if for almost $\sigma\in {\T^r}$ the system  $T(E(\mathcal A))(\sigma):=\{T(L_{k}\phi)(\sigma):~~ \phi\in \mathcal A, \ k\in \Gamma'\}$ constitutes a frame (Bessel) family for $J(\sigma)$ with the unified constants. 
\end{theorem}

\begin{proof}
Let $f\in L^2(N)$. Since $\|Tf\|=\|f\|$  we have  for each $\gamma \in \Gamma$,
$$
\begin{aligned}
\sum_{\phi\in \mathcal A, \gamma\in \Gamma} |\langle f, L_\gamma \phi\rangle |^2 
&= \sum_{\phi\in \mathcal A, \gamma\in \Gamma} \left|\langle Tf, T(L_\gamma \phi)\rangle \right|^2\\
&= \sum_{\phi\in \mathcal A, \gamma\in \Gamma} \left|\int_{\TT^r} \langle Tf(\sigma), T(L_\gamma \phi(\sigma)\rangle \ d\sigma \right|^2\\
&= \sum_{\phi\in \mathcal A, \gamma\in \Gamma} \left|\int_{\TT^r}\langle Tf(\sigma), \tilde\pi_\sigma(\gamma) T\phi(\sigma)\rangle \ d\sigma \right|^2.
\end{aligned}
$$
Writing $\gamma=km$, with $k \in \Gamma_1, m \in \G_0$, we get
$$
\begin{aligned}
  \sum_{\phi\in \mathcal A, \gamma\in \Gamma} \left|\int_{\T^r}\langle Tf(\sigma), \tilde\pi_\sigma(\gamma) T\phi(\sigma)\rangle \ d\sigma\right|^2&= 
    \sum_{\phi\in \mathcal A, (k,m)\in \Gamma} \left|\int_{\T^r} \langle T f(\sigma), 
    e^{2\pi i  \langle \sigma,m\rangle}\tilde\pi_\sigma(k) T\phi(\sigma)\rangle \ d\sigma\right|^2\\
    &= \sum_{\phi\in \mathcal A, (k,m)\in \Gamma} \left|\int_{\T^r} \langle Tf(\sigma),  \tilde\pi_\sigma(k) T(\phi)(\sigma)\rangle e^{-2\pi i  \langle \sigma,m\rangle}\ d\sigma\right|^2
\end{aligned}
$$
For each $k$  put  $G_k(\sigma):= \langle Tf(\sigma),  \tilde\pi_\sigma(k) T\phi(\sigma)\rangle$. Then $G_k$ is integrable with square-summable Fourier coefficients, therefore $G_k$ lies in $L^2({\T^r})$. By Fourier inversion we then continue the above as follows:
$$
\begin{aligned}
  \sum_{\phi\in \mathcal A, (k,m)\in \Gamma} \left|\int_{\T^r}\langle Tf(\sigma),  \tilde\pi_\sigma(k) T(\phi)(\sigma)\rangle e^{-2\pi i \langle \sigma,m\rangle}\ d\sigma\right|^2 
  & =  \sum_{\phi\in \mathcal A, (k,m)\in \Gamma} | \hat G_k(m)|^2\\
    &=   \sum_{\phi\in \mathcal A, k\in \Gamma'} \|G_{k}\|^2\\
    &=   \sum_{\phi\in \mathcal A, k\in \Gamma'}  \int_{\T^r}|G_k(\sigma)|^2 d\sigma
\end{aligned}  
$$
By substituting back $G_k(\sigma):= \langle Tf(\sigma),  \tilde\pi_\sigma(k) T(\phi)(\sigma)\rangle$ in above we obtain 
\begin{equation}
\begin{aligned} \label{sumid}
\sum_{\phi\in \mathcal A, \gamma\in \Gamma} |\langle f, L_\gamma \phi\rangle |^2 
&=  \sum_{\phi\in \mathcal A, k\in \Gamma'}  \int_{\T^r} |G_{k}(\sigma)|^2 d\sigma \\
&=  \sum_{\phi\in \mathcal A, k\in \Gamma'}  \int_{\T^r}  |\langle Tf(\sigma),  \tilde\pi_\sigma(k) T(\phi)(\sigma)\rangle|^2 d\sigma \\
     &=     \int_{\T^r}  \sum_{\phi\in \mathcal A, k\in \Gamma'}   |\langle Tf(\sigma),   T(L_k\phi)(\sigma)\rangle|^2 d\sigma.
\end{aligned}  
\end{equation}
Now suppose that $f \in \S$ and that for some $0 < A \le B<\infty$, the system   $T(E(\mathcal A))(\sigma)$ is an $(A,B)$-frame for a.e. $\sigma \in \T$. Then $Tf(\sigma) \in J(\sigma)$ holds for a.e. $\sigma$, so
$$
A\|(Tf)(\sigma)\|^2\leq   \int_{\T^r}  \sum_{\phi\in \mathcal A, k\in \Gamma'}   |\langle Tf(\sigma),   T(L_k\phi)(\sigma)\rangle|^2 d\sigma  \leq B\|(Tf)(\sigma)\|^2~.
$$
holds for a.e. $\sigma$. Integrating yields
$$
\begin{aligned}
A \|f\|^2 = A \| Tf\|^2 &= A \int_{\T^r} \|(Tf)(\sigma)\|^2 d\sigma 
\le \int_{\T}  \sum_{\phi\in \mathcal A, k\in \Gamma'}   |\langle Tf(\sigma),   T(L_k\phi)(\sigma)\rangle|^2 d\sigma  \\
&\le B \int_\T \|(Tf)(\sigma)\|^2d\sigma\\
&= B \|f\|^2. 
\end{aligned}
$$
By substituting (\ref{sumid}) we obtain 
$$
A \|f\|^2 \le\sum_{\phi\in \mathcal A, \gamma\in \Gamma} |\langle f, L_\gamma \phi\rangle |^2  \le B \|f\|^2. 
$$

  Now assume that $E(\mathcal A)$ is a frame family with constants $A$ and $B$. Let $\mathcal D$ be a countable dense subset of $\mathcal L$. To prove that the family  
  $T(E(\mathcal A))(\sigma)$ constitutes  a frame for $J(\sigma)$   for almost every $\sigma$, it is sufficient to show that for each $h\in \mathcal D$,
  \begin{equation}\label{inequality-for-dense-set}
A\|P_\sigma h \|^2\leq \sum_{\phi\in \mathcal A,\ k\in\Gamma} |\langle T(L_{k}\phi)(\sigma), h\rangle|^2 \leq B\|P_\sigma h \|^2  
  \end{equation}
  holds for a.e. $\sigma \in \T$, 
 where $P_\sigma$ is the orthogonal projection operator from $\mathcal L$ onto $J(\sigma)$. 
 (For, if  (\ref{inequality-for-dense-set}) holds for any $h\in \mathcal D$, then we have   $N_h$ a measure zero subset in ${\T}$ such that  for any $\sigma\in N_h^c$ the relation 
  (\ref{inequality-for-dense-set}) holds. Put $N= \cup_{h\in \mathcal D} N_h$. Then $N$ has measure zero and (\ref{inequality-for-dense-set}) holds for all $h\in \mathcal D$ and all $\sigma \in {\bf T}/N$. Since $P_\sigma(\mathcal D)$ is dense in $J(\sigma)$ for all $\sigma\in N^c$,  the assertion holds, i.e.,  $T(E(\mathcal A))(\sigma)$ constitutes  a frame for $J(\sigma)$ with the identical constants for all $\sigma\in N^c$.)
 To complete the proof, we still need to show  (\ref{inequality-for-dense-set}). For this, we assume that it fails for  some $h_0\in\mathcal D$ and define $G(\sigma):= \sum_{\phi\in \mathcal A,\ k\in\Gamma} |\langle T(L_{k}\phi)(\sigma), h_0 \rangle|^2$.  
 Then one of the following sets must have  positive measure. 
 $$
 \{\sigma\in {\TT^r}:~ G(\sigma)>B\|P_\sigma h_0\|^2\}, ~~ ~~~~\{\sigma\in {\T^r}:~ G(\sigma)<A\|P_\sigma h_0\|^2\}.
 $$
 Without lose of generality, we assume that the measure of the first set is positive. Therefore for some $\epsilon>0$, the measure of 
$ C:=\{\sigma\in {\T^r}:~ G(\sigma)>(B+\epsilon)\|P_\sigma h_0\|^2\}$ is positive too. 
Put $\tilde h_0(\sigma)= \chi_{C}(\sigma) P_\sigma h_0$. Then $\sigma \mapsto \tilde h_0(\sigma)$ is measurable and $\tilde h_0(\sigma)\in J(\sigma)$ and $\tilde h_0\in T(\S)=\int_{\bf T} J(\sigma)d\sigma$. Therefore  for some $f_0\in \S$,  $T(f_0)=\tilde h_0$. We show that the upper frame inequality does not hold for $f_0$ which leads us to a contradiction:
Let $\gamma=km$, with $k \in \Gamma_1, m \in \G_0$. Then
\begin{align}\label{first-line}
\sum_{\phi\in \mathcal A, \ \gamma\in \Gamma}
|\langle f_0, L_\gamma \phi\rangle|^2&= \sum_{\phi\in \mathcal A, \ \gamma\in \Gamma}
|\langle T(f_0), T(L_\gamma \phi)\rangle|^2\\\notag
&=\sum_{\phi\in \mathcal A, \ \gamma=(m,k)\in \Gamma}
|\int_{\TT^r}\langle T(f_0)(\sigma), e^{2\pi i  \langle \sigma,m\rangle}\tilde\pi_\sigma(k)T\phi\rangle d\sigma|^2\\\notag
&= \sum_{\phi\in \mathcal A, \ \gamma=(k,m)\in \Gamma}
|\hat G_{k}(m)|^2
\end{align}
where $G_{k}(\sigma)=  \langle T(f_0)(\sigma),  \tilde\pi_\sigma(k)T(\phi)\rangle$. For any $k$  we have 
$\sum_m |\hat G_{k}(m)|^2 = \|G_{k}\|^2$. By using  this  equality in our previous calculations and 
  using the Plancherel theorem and substituting back the function $G_k$, all together  we arrive the following:
\begin{align}\notag
  \sum_{\phi\in \mathcal A, \ k\in \Gamma'}\sum_m
|\hat G_{k}(m)|^2&= \sum_{\phi\in \mathcal A, \  k\in \Gamma'}\|G_{k}\|^2\\\notag
&=  \sum_{\phi\in \mathcal A, \  k\in \Gamma'}\int_{\T^r} | \langle T(f_0)(\sigma),  \tilde\pi_\sigma(k)\ T\phi(\sigma)\rangle|^2 d\sigma\\\notag
&=  \sum_{\phi\in \mathcal A, \  k\in \Gamma'}\int_{\T^r} | \langle  \tilde  h_0(\sigma),  \tilde\pi_\sigma(k)\ T\phi(\sigma)\rangle|^2 d\sigma\\\label{for substitution}
&= \int_C \sum_{\phi\in \mathcal A, \   k\in \Gamma'}  | \langle P_\sigma  h_0,  \tilde\pi_\sigma(k)\ T\phi(\sigma)\rangle|^2 d\sigma
\end{align}
By substituting $G(\sigma)= \sum_{\phi\in \mathcal A, \   k\in \Gamma'}  | \langle P_\sigma  h_0,  \tilde\pi_\sigma(k)\ T\phi(\sigma)\rangle|^2 $ in  the above, we have
\begin{align}\notag
(\ref{for substitution}) & > (B+\epsilon) \int_C \|P_\sigma h_0\|^2 d\sigma\\\notag
&= (B+\epsilon) \int_{\T^r} \|\chi_C(\sigma)P_\sigma h_0\|^2 d\sigma\\\notag
&= (B+\epsilon) \int_{\TT^r} \| \tilde h_0(\sigma)\|^2 d\sigma\\\notag
&= (B+\epsilon)  \| \tilde h_0 \|^2\\\notag
&= (B+\epsilon)  \|T^{-1} \tilde h_0 \|^2\\\label{last-line}
&= (B+\epsilon)  \|f_0 \|^2. 
\end{align}
Now a combination of (\ref{first-line}) and (\ref{last-line})
contradicts the assumption and hence we are done. 
 \end{proof}

\subsection{Riesz Bases}

Before we start with the characterization of Riesz bases obtained from $\Gamma$ shifts of a countable sets in terms of range functions, we shall introduce the following notation:  for any  finite supported sequence  $a = \{a_m\}_m\in l^2(\ZZ^{r})$, denote by $P_a$ the associated 
  trigonometric polynomial $P_a(\sigma)= \sum_m a_m e^{2\pi i \langle\sigma,m\rangle}$ for all $\sigma\in \TT^r$. Observe that 
\begin{equation}\label{isometry}
  \|P_a\|^2_2 = \sum_m |a_m|^2.
 \end{equation}

  We have the following. 
  
\begin{lemma}\label{equality}
Let $\mathcal A$ be a countable  set   in $L^2(N)$. Let $a = \{a_{\phi,k,m}\}_{\phi \in \mathcal A, km \in \Gamma}$ be a finitely supported sequence in $l^2(\mathcal A \times \Gamma)$. For each $\phi\in \mathcal A$ and $k \in \Gamma_1$ put $P_{\phi,k}(\sigma) =  \sum_{m} a_{\phi, k,m}  e^{2\pi i \langle\sigma,m\rangle}$.
 
 Then 
 $$
 \left\|\sum_{\phi\in \mathcal A, km\in \Gamma} a_{\phi, k,m} L_{km}\phi\right\|^2= \int_{\T^r} \left\|\sum_{\phi,k}P_{\phi,k}(\sigma)~\tilde \pi_\sigma(k)T(\phi)(\sigma)\right\|_\mathcal L^2 ~d\sigma
$$
 \end{lemma}
 \begin{proof}

 The proof is  based on some elementary calculations as follows: Since $T$ is isometric, we have 

$$
\begin{aligned}
\left\|\sum_{\phi\in \mathcal A, k,m} a_{\phi, k,m} L_{km}\phi\right\|_{L^2(N)}^2
 &= 
 \left\|\sum_{\phi\in \mathcal A, k,m} a_{\phi, k,m} T(L_{km}\phi)\right\|_{L^2({\T^r},\mathcal L)}^2\\
 &= \int_{\T^r}  \left\|\sum_{\phi\in \mathcal A,k,m} a_{\phi, k,m} T(L_{km}\phi)(\sigma)\right\|_{\mathcal L}^2 ~d\sigma\\
 &= \int_{\T^r} \left\|\sum_{\phi\in \mathcal A, k,m } a_{\phi, k,m} T(L_{km}\phi)(\sigma)\right\|_{\mathcal L}^2 ~d\sigma\\
 &= \int_{\T^r}  \left\|\sum_{\phi\in \mathcal A, k} \left(\sum_m a_{\phi, k,m} e^{2\pi i\langle \sigma,m\rangle}\right) \tilde\pi_\sigma(k)T(\phi)(\sigma)\right\|_{\mathcal L}^2 ~d\sigma\\
 &= \int_{\T^r} \left\|\sum_{\phi\in \mathcal A, k} P_{\phi,k}(\sigma)\tilde\pi_\sigma(k)T\phi (\sigma)\right\|_{\mathcal L}^2 ~d\sigma.
 \end{aligned}
$$
\end{proof}

\begin{proposition}\label{from babies to master}  
Let $\mathcal A$ be a countable subset of $L^2(N)$ and $E(\mathcal A)$ be  the set of $\Gamma$-translates of elements in $\mathcal A$ with associated range function $J$. Assume that for some $0 < A \le B<\infty$, $\{\tilde \pi_{\sigma}(k) T(\phi)(\sigma):~~ k\in \Gamma_1, ~ \phi\in \mathcal A\}$ is a Riesz basis  for $J(\sigma)$ with constants $A$ and $B$,  for almost every $\sigma\in {\T}$. Then   $E(\mathcal A)$  is a Riesz basis for its span with the same constants. 
\end{proposition}

\begin{proof}
Let a finitely supported sequence $\{a_{\phi,k,m}\}_{\phi\in\mathcal A, km\in\Gamma}$ be given. For each $\phi, k$, let $P_{\phi,k}$ be the trigonometric polynomial defined in Lemma \ref{equality}. Then  with  our assumptions, for almost every $\sigma\in {\T^r}$ and the  finite supported sequence $\{P_{\phi,k}(\sigma)\}_{\phi,k}=:\{b_{\phi,k}\}$ 
\begin{align}\label{riesz for $J$}
A\sum_{\phi,k} |P_{\phi,k}(\sigma)|^2 \leq \left\| \sum_{\phi,k} P_{\phi,k}(\sigma) \tilde \pi_\sigma(k) T(\phi)(\sigma) \right\|^2
\leq B \sum_{\phi,k} |P_{\phi,k}(\sigma)|^2 
\end{align}
where all the sums in above run on a finite index set of $\phi, k$. 
Integrating (\ref{riesz for $J$}) over ${\T^r}$ yields 
$$
A\sum_{\phi,k} \|P_{\phi,k}\|^2 \leq \int_{\T^r} \left\| \sum_{\phi,k} P_{\phi,k}(\sigma) \tilde \pi_{\sigma}(k) T\phi(\sigma) \right\|^2 ~d\sigma
\leq B \sum_{\phi,k} \|P_{\phi,k}\|^2  
$$
By applying Lemma \ref{equality} to the above middle sum  and substituting  (\ref{isometry})  we arrive at 
\begin{align}
A\sum_{\phi,k,m} |a_{\phi,k,m}|^2 \leq   \left\| \sum_{\phi,k,m} a_{\phi,k,m}~  L_{km}\phi\right\|^2  
\leq B \sum_{\phi,k,m} |a_{\phi,k,m}|^2
\end{align}
as desired. 
\end{proof}

\begin{theorem}\label{master and baby subspaces}
The set $E(\mathcal A)$ is a Riesz family with constants $A$ and $B$ if and only if 
 $\{\tilde \pi_\sigma(k)T(\phi)(\sigma):~~k\in \Gamma_1, ~ \phi\in \mathcal A\}$ is a Riesz family with constants $A$ and $B$, for a.e. $\sigma \in \T$. 
 
 \end{theorem}


\begin{proof} The proof is similar to those of \cite[Theorem 2.3, part (ii)]{B}; see also \cite[Theorem 4.3]{CP}. 

Choose any $a = \{a_{\phi, (k,m)}\}_{\phi \in \mathcal A, (k,m) \in \Gamma_1}$ having finite support.  For each $\phi, k$, let $P_{\phi,k}$ be the trigonometric polynomial defined in Lemma \ref{equality}. 

For sufficiency, suppose that $\{\tilde \pi_\sigma(k)T(\phi)(\sigma):~~k\in \Gamma_1, ~ \phi\in \mathcal A\}$ is a Riesz family with constants $A$ and $B$, for a.e. $\sigma \in \T$. Then for almost every $\sigma\in {\T}$ 
\begin{align}\label{riesz for $J$}
A\sum_{\phi,k} |P_{\phi,k}(\sigma)|^2 \leq \left\| \sum_{\phi,k} P_{\phi,k}(\sigma) \tilde \pi_\sigma(k) T(\phi)(\sigma) \right\|^2
\leq B \sum_{\phi,k} |P_{\phi,k}(\sigma)|^2 
\end{align}
where all the sums in above run on a finite index set of $\phi, k$. 
Integrating (\ref{riesz for $J$}) over ${\T}$ yields 
$$
A\sum_{\phi,k} \|P_{\phi,k}\|^2 \leq \int_{\T} \left\| \sum_{\phi,k} P_{\phi,k}(\sigma) \tilde \pi_{\sigma}(k) T\phi(\sigma) \right\|^2 ~d\sigma
\leq B \sum_{\phi,k} \|P_{\phi,k}\|^2  
$$
By applying Lemma \ref{equality} to the above middle sum  and substituting  (\ref{isometry})  we arrive at
\begin{align}
A\sum_{\phi,k,m} |a_{\phi,k,m}|^2 \leq   \left\| \sum_{\phi,k,m} a_{\phi,k,m}~  L_{km}\phi\right\|^2  
\leq B \sum_{\phi,k,m} |a_{\phi,k,m}|^2
\end{align}
as desired. 

On the other hand, suppose that  $E(\mathcal A)$ is a Riesz family in $L^2(N)$ with constants $A$ and $B$. Then for any finitely supported $\{a_{\phi,k,m}\}_{\phi\in \mathcal A, (k,m)\in \Gamma}$

\begin{align}\label{master inequality}
A\sum_{\phi,k,m} |a_{\phi,k,m}|^2\leq \left\|\sum_{\phi,k,m} a_{\phi,k,m} L_{km} \phi\right\|^2\leq B \sum_{\phi,k,m} |a_{\phi,k,m}|^2.
\end{align}
 Using Lemma \ref{equality} and (\ref{isometry}) this becomes
\begin{equation}\label{trig poly formula}
 A \int_{\T^r} \sum_{\phi,k} |P_{\phi,k}(\sigma)|^2 d\sigma \leq \int_{\T^r} \left\|\sum_{\phi,k}P_{\phi,k}(\sigma)~\tilde \pi_\sigma(k)T(\phi)(\sigma)\right\|_\mathcal L^2 ~d\sigma\leq B \int_{\T^r} \sum_{\phi,k}  |P_{\phi,k}(\sigma)|^2 d\sigma.
\end{equation}
Observe that (\ref{trig poly formula}) holds for any finite collection $\{P_{\phi,k}\}$ of trigonometric polynomials. (Note that this is correct since we can take $\{a_{\phi,(k,m)}\}$ in the definition of  $P_{\phi,k}$ on $\TT^r$ and plug the sequence back in (\ref{master inequality}) and then get  (\ref{trig poly formula}).)

Moreover, we can strengthen (\ref{trig poly formula})  as follows. Let $F$ be any finite subset of $\mathcal A \times \Gamma_1$ and for each $(\phi,k) \in F$ let $m_{\phi,k}$ be a bounded and measurable function on $\TT^r$. Then by Lusin's Theorem and density of trigonometric polynomials, for each $(\phi,k) \in F$, we have a sequence of trigonometric polynomials $(P_{\phi,k}^{i})_{i\in \N}$ such that 
 $$
 \| P_{\phi,k}^i\|_\infty \le \|m_{\phi,k}\|_\infty, \ \ \text{ for every } i \in \N, 
 $$
 and $P^i_{\phi,k} \rightarrow m_{\phi,k}$ a.e..  Note that (\ref{trig poly formula}) holds for $P^i_{\phi,k}$ for all $i\in \NN$. Hence by the Lebesgue Dominated Convergence Theorem we have that 
\begin{equation}\label{strong ineq}
 A \int_{\T^r} \sum_{\phi,k } |m_{\phi,k}(\sigma)|^2 d\sigma \leq \int_{\T^r} \left\|\sum_{\phi,k}m_{\phi,k}(\sigma)~\tilde \pi_\sigma(k)T(\phi)(\sigma)\right\|_\mathcal L^2 ~d\sigma\leq B \int_{\T^r} \sum_{\phi,k}  |m_{\phi,k}(\sigma)|^2 d\sigma
\end{equation}
 holds. 
 
 Now we must show that for a.e. $\sigma \in \T^r$, the following holds for  any $b = \{b_{\phi,k}\} \in l^2(\mathcal A \times \Gamma_1)$. 
$$ A \sum_{\phi,k} |b_{\phi,k}|^2 \le \left\| \sum_{\phi,k} b_{\phi,k} \tilde\pi_\sigma(k) T\phi(\sigma) \right\|_{\mathcal L}^2 \le B \sum |b_{\phi,k}|^2. 
 $$
    Now let $D$ be a countable dense subset of $ l^2(\mathcal A \times \Gamma_1)$ consisting of sequences of finite support. It is enough to show that for each $d \in D$, there is a conull measurable subset $E$ of $\T^r$ such that 
 $$
 A \sum_{\phi,k} |b_{\phi,k}|^2 \le \left\| \sum_{\phi,k} b_{\phi,k} \tilde\pi_\sigma(k) T\phi(\sigma) \right\|^2_{\mathcal L} \le B \sum |b_{\phi,k}|^2
 $$
 holds for every $\sigma \in E$. Suppose that this is false. Then there is a measurable subset $G$ of $\T^r$ with $|G| > 0$ and $\epsilon > 0$ such that at least one of the following holds:
 \begin{equation}\label{cont1}
 \left\| \sum_{\phi,k} b_{\phi,k} \tilde\pi_\sigma(k) T\phi(\sigma) \right\|^2_{\mathcal L} > (B + \epsilon) \sum_{\phi,k} |b_{\phi,k}|^2
 \end{equation}
 \begin{equation}\label{cont2}
 \left\| \sum_{\phi,k} b_{\phi,k} \tilde\pi_\sigma(k) T\phi(\sigma) \right\|^2_{\mathcal L} < (A - \epsilon) \sum_{\phi,k} |b_{\phi,k}|^2.
 \end{equation}
 Suppose that (\ref{cont1}) is the case, and consider the functions $m_{\phi,k} = b_{\phi,k} \mathbf 1_G$. Then 
 $$
 \begin{aligned}
  \int_{\T^r} \left\| \sum_{\phi,k} m_{\phi,k}(\sigma) \tilde\pi_\sigma(k) T\phi(\sigma) \right\|^2_\mathcal L d\sigma
 &=  \int_G \left\| \sum_{\phi,k} b_{\phi,k} \tilde\pi_\sigma(k) T\phi(\sigma) \right\|^2_\mathcal L d\sigma \\
&\ge (B + \epsilon) |G| \sum_{\phi,k} |b_{\phi,k}|^2 \\
&= (B+\epsilon) \int_{\T^r} \sum_{\phi,k} |m_{\phi,k}(\sigma)|^2 d\sigma
\end{aligned}
$$
which contradicts (\ref{strong ineq}). If  (\ref{cont2}) holds, we get a contradiction with using an analoguous  argument. 

\end{proof}

Since a Parseval frame which is also a Riesz basis must be orthonormal, we have the following.

 \begin{corollary} 
 Let $\phi\in L^{2}\left(  N\right)  $ where $N$ is SI/Z group. The system
$\left\{  L_{\gamma}\phi:~ \gamma=km\in\Gamma\right\}  $ is orthonormal iff for a.e. $\sigma$ 
the system $ \{\widetilde{\pi}_{\sigma}\left(
k\right)  T\phi\left(  \sigma\right):~  k\in   \Gamma_1\}$  is orthogonal and $\|T\phi(\sigma)\|_{\mathcal L}=1$ 

\end{corollary}


  \section{ Examples and Applications}\label{Examples and Applications}

  {\bf Example 1.}{ \it  The Heisenberg group}.  
Let $N$ denote the $3$ dimensional Heisenberg group. We choose a basis $\{X_1, X_2, X_3\}$ for its Lie algebra $\n$ where $[X_3,X_2]=X_1$ and other Lie brackets are zero. Thus the center of $N$ is $Z = \exp \R X_1$ and $\mathcal X = \exp \R X_2 \exp \R X_3$. With the coordinates $(x,y,z)= \exp(yX_2)\exp(xX_3)\exp(zX_1)$, we have $(x,y,z)\cdot (x',y',z') = (x+x', y+y', z + z' + xy')$. 

As is well-known, $N$ is an $SI/Z$ group and when $\z^*$ is identified with $\R$ as above, then $\Sigma = \mathbb{R}^{\ast}.$ and $|\mathbf{Pf}(\l)| = |\l|, \l \in \R^*$. It is also well-known that the Schr\"odinger representations act by a translation followed by a modulation. Specifically, for each $\l \in \R^*$, take the maximal subordinate subalgebra $\p(\l) = \R$-span$\{X_1, X_2\}$; then the corresponding irreducible induced representation acts via natural isomorphisms on functions $f$ in  $L^2(\R)$ by
$$
\bigl(\pi_\l(x,y,z) f\bigr)(t) = e^{2\pi i \l z} e^{-2\pi i \l y t} f(t - x), \ (x,y,z) \in N.
$$
Thus $\pi_\l(x,y,0) = M_{\l y} L_x$, where $L_xf(t) = f(t-x)$ and $M_vf(t) = e^{-2\pi iv t} f(t)$.





Thus for $\phi \in  L^1(N) \cap L^2(N)$, we have 
$$
\mathcal F \bigl(L_{(x,y,0)}\phi)\bigr)(\l) = M_{\l y} L_x \hat\phi(\l), \ \l \in \R^*.
$$


 
  
With above notations, we have $T(\phi)(\sigma)_j= |\sigma+j|^{1/2} \hat\phi(\sigma+j)$ for all $\phi\in L^1(N)\cap L^2(N)$, $\sigma\in \TT$, and $j\in \ZZ$. For any $a\in L^2({\T}, \mathcal L)$ and $\lambda$,  let $\tilde{\pi}_{\lambda}(k,l,m)a(\lambda):=\left( \pi_{\lambda+j}(k,l,m)\circ a(\lambda)_j\right)_j.$ Therefore 

 \begin{align}\label{intertwining-translation}
 T(L_{(k,l,m)}\phi)(\lambda)= \tilde{\pi}_\lambda(k,l,m)T(\phi)(\lambda).
 \end{align} 
 
 Take $\Gamma_1 = \exp \Z X_2 \exp \Z X_3$, so that, when $N$ is identified with $\R^2 \times \R$ by the coordinates above, $\Gamma$ is identified with the integer lattice. It is shown in \cite{F} that if $\H$ is a left-invariant subspace of $L^2(N)$ for which there is a Parseval frame of the form $L_\gamma \psi, \gamma \in \G$, then for all $\phi \in L^1(N) \cap L^2(N)$, $\mathcal F \phi$ is supported in $[-1,1]$. Thus in this case
 $T\phi(\sigma)_j = 0$ for all $j \le  -2$, and $j \ge 1$. See \cite{CM2} for an explicit example of a left-invariant subspace $\H$ and an orthonormal basis for $\H$ of the form $L_\gamma \psi, \gamma \in \G$.


 \vspace{.5cm}

\noindent
{\bf Example 2.} {\it A six-dimensional two step case }. Let $N$ be an $SI/Z$ with the following matrix realization.

\begin{align*} N=
\left\{  \left(
\begin{array}
[c]{ccccccc}%
1 & 0 & x_{2} & x_{1} & -y_{2} & -y_{1} & 2z_{2}-x_{1}y_{1}-x_{2}y_{2}\\
0 & 1 & x_{1} & x_{2} & -y_{1} & -y_{2} & 2z_{1}-x_{1}y_{2}-x_{2}y_{1}\\
0 & 0 & 1 & 0 & 0 & 0 & y_{2}\\
0 & 0 & 0 & 1 & 0 & 0 & y_{1}\\
0 & 0 & 0 & 0 & 1 & 0 & x_{2}\\
0 & 0 & 0 & 0 & 0 & 1 & x_{1}\\
0 & 0 & 0 & 0 & 0 & 0 & 1
\end{array}
\right)  :\left(
\begin{array}
[c]{c}%
z_{2}\\
z_{1}\\
y_{2}\\
y_{1}\\
x_{2}\\
x_{1}%
\end{array}
\right)  \in\mathbb{R}^{6}\right\}  ,
\end{align*}  with Lie algebra $\mathfrak{n}$ spanned by the vectors
$Z_{1},Z_{2},Y_{1},Y_{2},X_{1},X_{2}$ with the following non-trivial Lie
brackets,%
\begin{align*}
\left[  X_{1},Y_{1}\right]    & =\left[  X_{2},Y_{2}\right]  =Z_{1}\\
\left[  X_{1},Y_{2}\right]    & =\left[  X_{2},Y_{1}\right]  =Z_{2}.
\end{align*}
Here, for $f\in L^{2}\left(  N\right)  \cap L^{1}\left(  N\right)  $ we have the group Fourier transform defined as follows.
$$
\widehat f\left(  \lambda\right)  =\int_{N}f\left(  n\right)
\pi_{\lambda}\left(  n\right)  dn \text{ for all } \lambda\in \Sigma,
$$ where $\Sigma$ is identified with an open dense subset of $\mathbb{R}^{2},$ and the Plancherel measure is given by 
$$
d\mu\left(  \lambda\right)  =d\mu\left(  \lambda_{1},\lambda_{2}\right)
=\left\vert \lambda_{1}^{2}-\lambda_{2}^{2}\right\vert d\lambda_{1}%
d\lambda_{2}.
$$
For each $\lambda\in\Sigma,$ $\pi_{\lambda}$ is a corresponding irreducible
representation acting in $L^{2}\left(\mathbb{R}^{2}\right)  $ such that,
\begin{align}
& \pi_{\lambda}\left(  z_1, z_2, y_1, y_2, x_1, x_2\right)  f\left(  t_{1},t_{2}\right)
\label{rep}\\
& =e^{2\pi i\left(  z_{1}\lambda_{1}+z_{2}\lambda_{2}\right)  }e^{-2\pi
i\left(  \left\langle t,M_{\lambda}y\right\rangle \right)  }f\left(
t_{1}-x_{1},t_{2}-x_{2}\right)  ,\nonumber
\end{align}
where
$$
t=\left(
\begin{array}
[c]{c}%
t_{1}\\
t_{2}%
\end{array}
\right)  ,y=\left(
\begin{array}
[c]{c}%
y_{1}\\
y_{2}%
\end{array}
\right)  \text{ and }M_{\lambda}=\left(
\begin{array}
[c]{cc}%
\lambda_{1} & \lambda_{2}\\
\lambda_{2} & \lambda_{1}%
\end{array}
\right)  .
$$
Let
\begin{align*}
\Gamma & =\exp\left(\mathbb{Z}Z_{1}+\mathbb{Z}Z_{2}\right)  \exp\left(\mathbb{Z}Y_{1}+\mathbb{Z}Y_{2}\right)  \exp\left(\mathbb{Z}X_{1}+\mathbb{Z}X_{2}\right), \text{ and}  \\
\Gamma_1 & =\exp\left(\mathbb{Z}Y_{1}+\mathbb{Z}Y_{2}\right)  \exp\left(\mathbb{Z}X_{1}+\mathbb{Z}X_{2}\right)  .
\end{align*}
A range function here is a mapping from $\mathbb{T}^{2}$ into $l^{2}\left(\mathbb{Z}^{2},L^{2}\left(\mathbb{R}^{2}\right)  \otimes L^{2}\left(\mathbb{R}^{2}\right)  \right)  .$ 

 Given a $\Gamma$-shift-invariant subspace of $L^2(N)$. Then  $\{L_\gamma\phi : \gamma \in \Gamma\}$ is a frame with frame constant $A,B$ for the closure of its $\CC$-span  if and only if $\{\widetilde{\pi }_{\sigma}(\gamma)T\phi :\gamma \in \Gamma\}$ forms a frame with the same constants for a closed subspace $\{ T(L_k\phi)(\sigma):~ k\in \Gamma_1\}$  of $\mathcal{L}= l^{2}\left(\mathbb{Z}^{2},L^{2}\left(\mathbb{R}^{2}\right)  \otimes L^{2}\left(\mathbb{R}^{2}\right)  \right)$ for a.e $\sigma \in \mathbb{T}$. \\

\noindent
{\bf Application.}  
We will give an example of a $\Gamma$-shift-invariant space of $L^2(N)$ which is not
left-invariant. We start by defining a specific range function
$$
J:\mathbb{T}^{2}\rightarrow\left\{  \text{closed subspaces of }l^{2}\left(\mathbb{Z}^{2},L^{2}\left(\mathbb{R}^{2}\right)  \otimes L^{2}\left(\mathbb{R}^{2}\right)  \right)  \right\}  ,
$$
such that for a.e $\sigma\in\mathbb{T}^{2},J\left(  \sigma\right)
=l^{2}\left(  I,\mathcal{H}_{\sigma}\right)  $ where $I=\left\{\left(  n,n\right)  : n \in \mathbb{Z}, |n|\leq 4\right\}  $ and,
$$
\mathcal{H}_{\sigma}\mathcal{=}\overline{\text{ }\mathcal{\mathbb{C}}\text{-span}}\text{ }\left\{
\begin{array}
[c]{c}%
\pi_{\sigma}\left(  k\right)  \circ\left(  f\otimes g\right)
:f=1_{\left[  0,1/2\right)  ^{2}}\\
g=1_{\left[  0,1\right)  ^{2}},\text{and }k\in\Gamma_1%
\end{array}
\right\}  .
$$
We consider $M_{J}$ as defined earlier. It follows that $S=T^{-1}\left(
M_{J}\right)$ is $\Gamma$-shift-invariant but not left-invariant simply
because the span of  $
\left\{  \pi_{\sigma}\left( k\right)  f: ~~ k\in\Gamma_1 \right\}$
is a proper subspace of $L^{2}\left(\mathbb{R}^{2}\right).$ In fact, given $\phi\in S=T^{-1}\left(  M_{J}\right) ,$ and referring to the action of the representation defined in (\ref{rep}),
$L_{\exp\left( \frac{X_{1}}{2}+\frac{X_{2}}{2}\right)}
\phi= L_{(\frac{1}{2}, \frac{1}{2},  0, 0, 0, 0)} \phi\notin S.$ \\

 {\bf Example 4.} {\it A 3-step case}. 
Let us consider a nilpotent Lie group $N$ with its Lie algebra $\mathfrak{n}$
spanned by the vectors $\left\{Z,Y_{2},Y_{1},X_{2},X_{1}\right\}  $ so
that we have the following non-trivial Lie brackets: 
$
\left[  X_{1},Y_{1}\right]  =Y_{2},\text{ }\left[  X_{1},Y_{2}\right]
=\left[  X_{2},Y_{1}\right]  =Z.
$
Observe that the center of $\mathfrak{n}$ is the $1$-dimensional vector space
spanned by $Z.$ It can be shown that $N$ is square integrable modulo the
center and that its unitary dual is a subset of the dual of $\mathbb{R}Z$ which can be identified with $\mathbb{R}^{\ast}.$ The Plancherel measure on $\mathbb{R}^{\ast}$ is up to multiplication by a constant equal to $\lambda^{2}d\lambda.$
For each $\lambda\in\mathbb{R}^{\ast},$ the corresponding unitary representation $\pi_{\lambda}$ acts in
$L^{2}\left(\mathbb{R}^{2}\right)  $ in the following ways.
\begin{align*}
\pi_{\lambda}\left(  \exp tZ\right)  F\left(  t_{1},t_{2}\right)    &
=e^{2\pi it\lambda}F\left(  t_{1},t_{2}\right)  \\
\pi_{\lambda}\left(  \exp\left(  y_{2}Y_{2}+y_{1}Y_{1}\right)  \right)
F\left(  t_{1},t_{2}\right)    & =e^{\pi i\lambda\left(  t_{1}^{2}y_{1}%
-2t_{1}y_{2}\right)  }e^{-2\pi i\lambda t_{2}y_{1}\lambda}F\left(  t_{1}%
,t_{2}\right)  \\
\pi_{\lambda}\left(  \exp\left(  x_{2}X_{2}+x_{1}X_{1}\right)  \right)
F\left(  t_{1},t_{2}\right)    & =F\left(  t_{1}-x_{1},t_{2}-x_{2}\right)  .
\end{align*}
We use the following exponential coordinates
$
\left(  x_{1},x_{2},y_{1},y_{2},z\right) $ $=\exp\left(  x_{1}X_{1}\right)
$ $\exp\left(  x_{2}X_{2}\right) $ $ \exp\left(  y_{1}Y_{1}\right)  \exp\left(
y_{2}Y_{2}\right)  \exp\left(  zZ\right).
$ We define 
$$T:L^{2}\left(  N\right)  \rightarrow L^{2}\left(  \TT, \mathcal L     \right)  $$
with $\mathcal L= l^{2}(\ZZ,   \mathcal{HS}(  L^{2}(\mathbb{R}^2)) )$  such that for almost every
$\sigma\in\TT$  we have $T\phi\left(  \sigma\right)  $ which
is a sequence of Hilbert Schmidt operators in $\mathcal{HS}\left(
L^{2}\left(\mathbb{R}^{2}\right)  \right).$ More precisely, $T\phi\left(  \sigma\right)
_{j}= \left\vert
\sigma+j\right\vert\left(  \mathcal{F\phi}\right)\left(  \sigma+j\right) $, $j\in \ZZ$,  and
$$
T\left(  L_{\left(  x_{1},x_{2},y_{1},y_{2},z\right)}  \phi\right)  \left(
\sigma\right)  _{j}=\widetilde{\pi}_{\sigma+j}\left(  x_{1},x_{2},y_{1},y_{2}%
,z\right)  T\phi\left(  \sigma+j\right)$$  with $T\phi\left(
\sigma+j\right)  \in\mathcal{HS}\left(  L^{2}\left(\mathbb{R}^{2}\right)  \right) .$  We define the following discrete
subsets of $N$.%
\begin{align*}
\Gamma & =\left\{  \left(  k_{1},k_{2},k_{3},k_{4},m\right)  \in N:~ k_i\in \ZZ ,m\in\mathbb{Z}\right\}, \text{ and}  \\
\Gamma_1  & =\left\{  \left(  k_{1},k_{2},k_{3},k_{4},0\right)  \in
N:~ k_i\in \mathbb{Z}\right\}  .
\end{align*}
Also, we define the range function $J$ such that for almost every $\sigma\in \T  ,$ $J\left(  \sigma\right)  $ is a closed subspace of
$L^{2}\left(  \T  ,l^{2}\left(\ZZ,   \mathcal{HS}\left(
L^{2}\left(\mathbb{R}^{2}\right)  \right)  \right)  \right)  .$ The following must hold.

\begin{enumerate}
\item $S$ is a $\Gamma$-shift-invariant subspace of $L^{2}\left(  N\right)  $
if and only if for almost every $\sigma \in \T$, $J(\sigma)$ is invariant under the action of
$\widetilde{\pi}_{\sigma}\left(  k_{1},k_{2},k_{3},k_{4},0\right)  $ for all
$\left(  k_{1},k_{2},k_{3},k_{4}\right)  \in \mathbb{Z}^{4}$. Furthermore 
$$
T\left(  S\right)  =\left\{  a\in L^{2}\left( \T
,l^{2}\left( \ZZ,  \mathcal{HS}\left(  L^{2}\left(\mathbb{R}^{2}\right)  \right)  \right)  \right)  :a\left(  \sigma\right)  \in J\left(
\sigma\right)  \right\}  .
$$
\item Let $\mathcal{A}$ be a countable set in $L^{2}\left(  N\right)  $. Define $S$ such
that
$$
S=\overline{\mathbb{C}\text{-}\:\mathrm{span}}\left\{  L_\gamma  \phi:~ \gamma\in\Gamma,\phi\in
\mathcal{A}\right\}  .
$$
The system $\left\{  L_\gamma \phi:\gamma\in\Gamma,\phi\in
\mathcal{A}\right\}  $ constitutes a frame (or a Riesz basis) in $S$ with constants 
$0<A\leq B<\infty$ if and only if for almost every $\sigma\in \T ,$ the system 
$$
\left\{  \widetilde{\pi}_{\sigma}\left(  k_{1},k_{2},k_{3},k_{4},0\right)
T\phi\left(  \sigma\right)  :\left(  k_{1},k_{2},k_{3},k_{4}\right)  \in \mathbb{Z}^{4},\phi\in \mathcal A\right\}
$$
constitutes a frame (or a Riesz basis) in $J\left(  \sigma\right)  $ with the
same frame constants. 
\end{enumerate}

\vspace{1cm}

 Bradley Currey , 
 {Department of Mathematics and Computer
Science, Saint Louis University, St. Louis, MO 63103}\\
 { \footnotesize{E-mail address: \texttt{{ curreybn@slu.edu}}}\\}

  Azita Mayeli, 
   {Mathematics Department, Queensborough C. College of City University of New York, Bayside, NY 11362}\\ 
  \footnotesize{E-mail address: \texttt{{amayeli@qcc.cuny.edu}}}\\

Vignon Oussa, 
{Department of Mathematics and Computer
Science, Saint Louis University, St. Louis, MO 63103}\\
 { \footnotesize{E-mail address: \texttt{{voussa@slu.edu}}\\}

\end{document}